\ifdefined\XeTeXversion\else\ifdefined\pdftexversion\else
\newcount\pdftexversion\pdftexversion=140

\fi\fi
\documentclass{amsart}
\usepackage{amsfonts,amssymb}
\usepackage{enumerate,graphicx}
\usepackage[labelfont=normal,labelformat=simple]{subcaption}
\usepackage[square,numbers]{natbib}
\usepackage{mathrsfs,bbm}
\usepackage{tikz,tikzscale,tikz-cd}
\usepackage{multirow,xparse,fp}
\usepackage{environ}
\usepackage{amstext}
\usepackage{hyperref}
\usepackage{diagbox}
\usepackage{float}

\newtheorem{theorem}{Theorem}[section]

\newtheorem{proposition}[theorem]{Proposition}

\newtheorem{lemma}[theorem]{Lemma}
\theoremstyle{definition}

\newtheorem{example}[theorem]{Example}

\newtheorem{problem}{Problem}
\newtheorem{remark}[theorem]{Remark}

\theoremstyle{plain} 
\newcommand{\thistheoremname}{}
\newtheorem*{genericthm*}{\thistheoremname}
\newenvironment{namedthm*}[1]
  {\renewcommand{\thistheoremname}{#1}%
   \begin{genericthm*}}
  {\end{genericthm*}}

\newcommand\Nint{\mathbb{N}}

\newcommand\Zint{\mathbb{Z}}
\newcommand\norm[1]{\lvert #1 \rvert}

\newcommand\treeshift[1]{\mathcal{T}_{#1}}

\usepackage{amsmath}
\usepackage{color}

\begin{document}

\title{Topological Entropy for Shifts of Finite Type Over $\mathbb{Z}$ and Trees}

\author[J-C Ban]{Jung-Chao Ban}
\address[Jung-Chao Ban]{Department of Mathematical Sciences, National Chengchi University, Taipei 11605, Taiwan, ROC.}
\address{Math. Division, National Center for Theoretical Science, National Taiwan University, Taipei 10617, Taiwan. ROC.}
\email{jcban@nccu.edu.tw}

\author[C-H Chang]{Chih-Hung Chang}
\address[Chih-Hung Chang]{Department of Applied Mathematics, National University of Kaohsiung, Kaohsiung 81148, Taiwan, ROC.}
\email{chchang@nuk.edu.tw}

\author[W-G Hu]{Wen-Guei Hu}
\address[Wen-Guei Hu]{College of Mathematics, Sichuan University, Chengdu, 610064, China}
\email{wghu@scu.edu.cn}

\author[Y-L Wu]{Yu-Liang Wu}
\address[Yu-Liang Wu]{Department of Applied Mathematics, National Chiao Tung University, Hsinchu 30010, Taiwan, ROC.}
\email{s92077.am08g@nctu.edu.tw}

\thanks{Ban and Chang are partially supported by the Ministry of Science and Technology, ROC (Contract No MOST 109-2115-M-004-002-MY2 and 109-2115-M-390-003-MY3). Hu is partially supported by the National Natural Science Foundation of China (Grant No.11601355).}

\date{April 24, 2022}

\baselineskip=1.2\baselineskip

\begin{abstract}
We study the topological entropy of hom tree-shifts and show that, although the topological entropy is not a conjugacy invariant for tree-shifts in general, it remains invariant for hom tree higher block shifts. In \cite{PS-TCS2018, PS-DCDS2020}, Petersen and Salama demonstrated the existence of topological entropy for tree-shifts and $h(\mathcal{T}_X) \geq h(X)$, where $\mathcal{T}_X$ is the hom tree-shift derived from $X$. We characterize a necessary and sufficient condition when the equality holds for the case where $X$ is a shift of finite type. Additionally, two novel phenomena have been revealed for tree-shifts. There is a gap in the set of topological entropy of hom tree-shifts of finite type, making such a set not dense. Last but not least, the topological entropy of a reducible hom tree-shift of finite type can be strictly larger than that of its maximal irreducible component.
\end{abstract}
\maketitle

\section{Introduction}

A $\mathbb{Z}^d$ shift space $X$ over alphabet $\mathcal{A}$ is a set of configurations that avoid any patterns labeled by $\mathcal{A}$ from some set $\mathcal{F}$; $X$ is a \emph{shift of finite type} (SFT) if $\mathcal{F}$ is finite. The study of $\mathbb{Z}^d$ SFTs is rife with numerous undecidability issues whenever $d \geq 2$. For instance, it is not even decidable whether $X$ is nonempty. We refer the reader to \cite{Berger-MAMS1966, Kari-DM1996, Robinson-IM1971}, for example.

Because of the complexity of dynamics of $\mathbb{Z}^d$ shift spaces as well as inspired by the physical models, the elucidation of \emph{hom-shifts} is imperative. A hom-shift is a nearest neighbor SFT, which is symmetric and isotropic; more explicitly, if $a, b \in \mathcal{A}$ are forbidden to sit next to each other in some coordinate direction, they are forbidden to sit next to each other in all coordinate directions. Many important SFTs, such as the hard square shift and the $n$-colored chessboard, arise as hom-shifts. Chandgotia and Marcus \cite{CM-PJM2018} investigated the mixing properties of hom-shifts and related them to some questions in graph theory therein.

Tree-shifts, which are shift spaces defined on free semigroups, have received extensive attention recently since they exhibit the natural structure of one-dimensional symbolic dynamics while equipped with multiple directional shift transformations. The dynamical phenomena of shift spaces on trees are fruitful since they constitute an intermediate class of symbolic dynamics between $\mathbb{Z}$ and $\mathbb{Z}^{d}$, $d\geq 2$. See \cite{AB-TCS2012, AB-TCS2013, BC-JMP2017, BC-TAMS2017, CCF+-TCS2013, FF-2012} and the references therein, for instance.

Aside from the elucidation of $\mathbb{Z}^d$ hom-shifts, investigating the hom tree-shifts is an alternative path on the study of abstract tree-shifts of finite type. It is to be noted that in this article we abuse the word `hom' to describe that a shift space admits rules (i.e., $\mathcal{F}$ in the case of $\mathbb{Z}^d$ shift space) concerning only the patterns on each infinite path and forgetting the directions in the underlying lattice; therefore, such a space consists of all the configurations each of whose infinite paths comes from a one-dimensional shift space, whether of finite type or not. Under the circumstances, the terms `hom Markov tree-shift' (see Section 1.2) and `hom tree-shift of finite type' (see Section 3) are used to emphasize the determining one-dimensional shift space is a nearest neighbor SFT and SFT, respectively. Mairesse and Marcovici \cite{MM-IJFCS2017} constructed a stationary Markov measure, a uniform measure, for a hom Markov tree-shift. Petersen and Salama \cite{PS-TCS2018, PS-DCDS2020} demonstrated that the topological entropy of arbitrary tree-shift exists as the infimum of the growth rate of $n$-blocks. They also revealed that the topological entropy of $X$ is not larger than the topological entropy of the hom tree-shift derived from $X$.

This paper, inspired by \cite{PS-TCS2018, PS-DCDS2020}, aims at the relations between $\mathbb{Z}$ SFT $X$ and its associated hom tree-shift $\mathcal{T}_X$, especially on the discussion of topological entropy. Theorem \ref{thm:irr_H_condition} gives a necessary and sufficient condition for the equality of topological entropy of $X$ and $\mathcal{T}_X$ when $X$ is Markovian. Although the topological entropy is not a conjugacy invariant for tree-shifts in general (see \cite{Bowen-2019} for a brief survey), Theorem \ref{thm:n_block_entropy_inv} reveals that the topological entropy remains the same for hom tree-shifts arose from higher block shifts; this, together with Theorem \ref{thm:irr_H_condition}, gives a criterion for the equality of topological entropy of $X$ and $\mathcal{T}_X$ for general SFTs. In the meantime, there are interesting phenomena, such as the existence of an SFT $X$ and $n \in \mathbb{N}$ such that $\mathcal{T}_X$ is not conjugate to $\mathcal{T}_{X^{[n]}}$ (Proposition \ref{prop:homTSFT-higher-block-not-conjugate}), there are SFTs $X, Y$ such that $X$ is conjugate to $Y$ but $h(\mathcal{T}_X) \neq h(\mathcal{T}_Y)$ (Proposition \ref{prop:counterexamples}), which remain to give more in-depth investigation.

Aside from the elucidation of relations between $X$ and $\mathcal{T}_X$, two novel phenomena in tree-shifts are also of interest. It is well-known that the set of topological entropy of SFTs is dense in the closed interval $[0, \log |\mathcal{A}|]$ (cf.~\cite{Des-IM2006} and the references therein) and the topological entropy of $X$ is attained by its maximal irreducible component (cf.~\cite{LM-1995}). Proposition \ref{prop:entropy_gap} specifies a gap in the set of topological entropy of hom tree-SFTs and Proposition \ref{prop:entropy_large_irred_component} addresses that $\mathcal{T}_X$ can have strictly larger topological entropy than that of any of its irreducible components in general.

\subsection{Notations and Definitions}
Let $d \geq 2$ and $\Sigma = \{0, 1, \ldots, d-1\}$. The $d$-tree $\Sigma^* = \cup_{n \geq 0} \Sigma^n$ is the set of all finite words on $\Sigma$ and is naturally visualized as the Cayley graph of the free semigroup on $d$ generators. The empty word $\epsilon$ is the only word of zero length and corresponds to the root of the tree and the identity element of the semigroup. Let $\mathcal{A}$ be a finite alphabet. A \emph{labeled tree} is a function $t: \Sigma^* \to \mathcal{A}$. For each node $w \in \Sigma^*$, $t_w := t(w)$ refers to the label attached to $w$. Denote by $\Delta_n := \cup_{i=0}^{n}\Sigma^i$ the initial subtree of the $d$-tree. An \emph{$n$-block} is a function $u : \Delta_n \rightarrow \mathcal{A}$. We say that a block $u$ appears in a labeled tree $t$ if there is a node $s \in \Sigma^*$ such that $t_{s w}=u_w$ for all $w \in \Delta_n$; otherwise, $t$ avoids $u$. A \emph{tree-shift} is a set $\treeshift{} \subseteq \mathcal{A}^{\Sigma^*}$ of labeled trees which avoid all of a certain set of forbidden blocks.

For each binary matrix $A$ indexed by $\mathcal{A}$, there is a \emph{Markov shift} $\mathsf{X}_A$ and a \emph{hom Markov tree-shift} $\mathcal{T}_A$ defined as
$$
\mathsf{X}_A = \{x \in \mathcal{A}^{\mathbb{N}}: A_{x_i, x_{i+1}} = 1 \text{ for all } i \in \mathbb{N}\}
$$
and
$$
\mathcal{T}_A = \{t \in \mathcal{A}^{\Sigma^*}: A_{t_w, t_{wi}} = 1 \text{ for all } w \in \Sigma^*, i \in \Sigma\},
$$
respectively. By writing $B_n (\mathsf{X}_A) := \{x_{[1,n]} \in \mathcal{A}^{n}: x \in \mathsf{X}_A\}$,
the set of admissible $n$-blocks of $\mathsf{X}_A$, the \emph{topological entropy of $\mathsf{X}_A$} is defined as
$$
h(\mathsf{X}_A) = \lim_{n \to \infty} \frac{\log |B_n(\mathsf{X}_A)|}{n},
$$
which measures the growth rate of admissible $n$-words concerning their support. Analogously, the set $B_n (\mathcal{T}_A)$ refers to the set of $n$-blocks appearing in $\mathcal{T}_A$ and \emph{the topological entropy of $\mathcal{T}_A$}
$$
h(\mathcal{T}_A) = \lim_{n \to \infty} \frac{\log |B_n(\mathcal{T}_A)|}{|\Delta_n|}
$$
is defined, where $\Delta_n = \cup_{i=0}^{n} \Sigma^i$.

\section{Hom Tree-Shift of Finite Type with Larger Topological Entropy}

Suppose $A$ is a $k$-by-$k$ binary matrix indexed by the alphabet $\mathcal{A}=\{1,2,\ldots,k\}$. Although $|\Delta_n| \gg n$, Peteren and Salama \cite{PS-TCS2018} revealed the quite interesting property that $h(\mathsf{X}_A) \leq h(\mathcal{T}_A)$. It is of interest whether there is a criterion for the equality. Let
$$
M:=\max_{i} \sum_j A_{i,j}
\quad \text{and} \quad
m:= \min_{i} \sum_j A_{i,j}
$$
be the maximal and minimal row sum of $A$, respectively. Theorem \ref{thm:irr_H_condition} yields a necessary and sufficient condition for determining when the equality holds.

\begin{theorem} \label{thm:irr_H_condition}
Suppose $\mathcal{T}_A$ is a hom Markov tree-shift induced by $A$ and $A$ is irreducible. Then, $h(\mathsf{X}_A)=h(\mathcal{T}_A)$ if and only if $M=m$.
\end{theorem}

In order to prove the theorem, the following two lemmas illustrate essential mechanisms for estimating the topological entropy of $\mathcal{T}_A$. For simplicity of notation, we denote by $\mathbf{x}(n) = (x_1(n), x_2(n), \ldots, x_k(n)) \in \mathbb{N}^k$ the vector of the numbers of patterns such that $x_i(n)=\norm{\{u \in B_n (\mathcal{T}_{A}): u_\epsilon=i\}}$. It is noteworthy that under the assumption that $A$ has no zero rows or zero columns (which is obviously valid if $A$ is irreducible), this vector $\mathbf{x}(n)$ can be iteratively computed as 
\begin{equation} \label{eq:iteration}
    x_i(n+1)=(A \mathbf{x}(n))_i^d:=\left(\sum_{j=1}^k A_{i,j} x_j(n)\right)^d
\end{equation}
with the initial condition $\mathbf{x}(0)=(1,1,\ldots,1)$. The following lemmas demonstrate that the condition $M>m$ yields an important subset of nonnegative integers $\Zint_+$ having positive lower Banach density. For nonnegative integers $n_1 \leq n_2$, denote by
$$
[n_1, n_2] = \{n \in \mathbb{Z}_+: n_1 \leq n \leq n_2\}
$$
the set of integers between $n_1$ and $n_2$.

\begin{lemma} \label{lem:universal_gap_property}
Suppose $A \in \{0,1\}^{k \times k}$ is an irreducible matrix with $M > m$, and for each $\delta \in (1,(\frac{M}{m})^{\frac{d}{d+1}})$ let 
$$
S=S(\delta):=\{n \in \mathbb{Z}_+: \max_{i_1,i_2} \frac{x_{i_1}(n)}{ x_{i_2}(n)} > \delta\}.
$$
Then, there exists $N \in \mathbb{N}$ such that $[n, n+N] \cap S \neq \varnothing$ for every $n \ge 0$.
\end{lemma}
\begin{proof}
Since $\delta \in (1,(\frac{M}{m})^{\frac{d}{d+1}})$, there exists $N \in \mathbb{N}$ such that $\left(\frac{M}{m} \delta^{-\frac{d+1}{d}}\right)^{N d} > \delta^2$. We prove this $N$ satisfy the requirement by contradiction. Suppose there exists $n \in \mathbb{N}$ such that $[n,n+N] \subset \mathbb{N} \setminus S$, i.e., 
\begin{equation*}
\max_{i_1,i_2} \frac{x_{i_1}(m)}{x_{i_2}(m)} = \frac{\max_{i} x_{i}(m)}{\min_{i} x_{i}(m)} \le \delta, \forall m \in [n,n+N].
\end{equation*}
Suppose $i_1^*, i_2^* \in \mathcal{A}$ such that $M = \sum_{j} A_{i_1^*,j}$ and $m = \sum_{j} A_{i_2^*,j}$. Then, applying \eqref{eq:iteration} yields
\begin{align*}
& \quad \frac{x_{i_1^*}(n+N)}{x_{i_2^*}(n+N)} = \left( \frac{\sum_{j=1}^k A_{i_1^*,j} x_{j}(n+N-1)}{\sum_{j=1}^k A_{i_2^*,j} x_{j}(n+N-1)} \right)^d \\
& \ge \left(\frac{M \cdot \min_{i} x_{i}(n+N-1)}{m \cdot \max_{i} x_{i}(n+N-1)}\right)^d \\
& = \left(\frac{M}{m}\right)^d \left(\frac{\min_{i} x_{i}(n+N-1)}{\max_{i} x_{i}(n+N-1)}\right)^{d-1} \left(\frac{\min_{i} x_{i}(n+N-1)}{\max_{i} x_{i}(n+N-1)}\right) \\
& \ge \left(\frac{M}{m}\right)^d \frac{1}{\delta^{d-1}} \left(\frac{\min_{i} x_{i}(n+N-1)}{x_{i_1^*}(n+N-1)}\right) \left(\frac{x_{i_2^*}(n+N-1)}{\max_{i} x_{i}(n+N-1)}\right) \left(\frac{x_{i_1^*}(n+N-1)}{x_{i_2^*}(n+N-1)}\right) \\
& \ge \left(\frac{M}{m} \delta^{-\frac{d+1}{d}}\right)^d \frac{x_{i_1^*}(n+N-1)}{x_{i_2^*}(n+N-1)} \\
& \ge \left(\frac{M}{m} \delta^{- \frac{d+1}{d}}\right)^{N d} \frac{x_{i_1^*}(n)}{x_{i_2^*}(n)} \ge \left(\frac{M}{m} \delta^{- \frac{d+1}{d}}\right)^{N d} \frac{1}{\delta} > \delta
\end{align*}
which is a contradiction. The proof is thus complete.
\end{proof}

\begin{remark} \label{rmk:pigeonhole}
For every $\delta \in (1,(\frac{M}{m})^{\frac{d}{d+1}})$ and $n \in S(\delta)$, express the components of $\mathbf{x}(n)$ in ascending order as
$$
x_{i_1}(n) \leq x_{i_2}(n) \leq \ldots \leq x_{i_k}(n).
$$
The pigeonhole principle indicates that there exists $1 \leq l \leq k-1$ such that $\frac{x_{i_{l+1}}(n)}{x_{i_{l}}(n)} > \delta^{\frac{1}{k-1}}$, or equivalently, $\frac{(A \mathbf{x}(n-1))_{i_{l+1}}}{(A \mathbf{x}(n-1))_{i_{l}}} > \delta^{\frac{1}{d (k-1)}}$.
\end{remark}

\begin{lemma} \label{lem:holder_strict_inequality}
Suppose $\{v_i\}_{i=1}^k \subset (0,1)$ is given such that $\sum_{i=1}^k v_i = 1$. There exists a strictly increasing function $C:[1, \infty) \to [1, \infty)$ satisfying the following:
\begin{enumerate}
    \item $C(1)=1$,
    \item $\frac{\sum_{i=1}^k w_i a_i^d}{(\sum_{i=1}^k w_i a_i)^d} \ge C(\frac{a_{l+1}}{a_l}) \ge 1$ for all $1 \le l \le k-1$ and $0 < a_1 \le a_2 \le \ldots \le a_k$,
\end{enumerate}
where $\{w_i\}_{i=1}^k$ is an arbitrary permutation of $\{v_i\}_{i=1}^k$.
\end{lemma}
\begin{proof}
Suppose $\mathbf{w} = \{w_i\}_{i=1}^k$ is a permutation of $\{v_i\}_{i=1}^k$. Consider the functions $J_{\mathbf{w}}(t_1, t_2, \ldots, t_k) := \sum_{i=1}^k w_i t_i^d$, $I_{\mathbf{w}}(t_1, t_2, \ldots, t_k) :=\sum_{i=1}^k w_i t_i$, and
$$
f_{\mathbf{w}}(t_1, t_2, \ldots, t_k):=\frac{J_{\mathbf{w}}(t_1, t_2, \ldots, t_k)}{(I_{\mathbf{w}}(t_1, t_2, \ldots, t_k))^d}.
$$
By H\"{o}lder's inequality with $p = d$, $q = d/(d-1)$, $w_i^{1/d}a_i$ for the first sequence and $w_i^{(d-1)/d}$ for the second, we have
$$
f_{\mathbf{w}}(a_1,a_2,\ldots,a_k) = \frac{\sum_{i=1}^k w_i a_i^d}{(\sum_{i=1}^k w_i a_i)^d} \ge \frac{\sum_{i=1}^k w_i a_i^d}{(\sum_{i=1}^k (w_i^{\frac{1}{d}} a_i)^d)(\sum_{i=1}^k (w_i^{\frac{d-1}{d}})^{\frac{d}{d-1}})^{d-1}} = 1.
$$
Note that the partial differentiation of $f_{\mathbf{w}}$ with respect to $t_j$ is
\[
\frac{\partial f_{\mathbf{w}}}{\partial t_j} (a_1,a_2,\ldots,a_k) = \frac{d \cdot I_{\mathbf{w}}^{d-1} w_j \cdot \sum_{i=1}^k w_i a_i (a_j^{d-1} - a_i^{d-1})}{I_{\mathbf{w}}^{2 d}}.
\] 
As a consequence, 
\[
\frac{\partial f_{\mathbf{w}}}{\partial t_k} (a_1,a_2,\ldots,t) \ge 0, \forall t \in (a_{k-1}, a_k),
\]
and
\[
\frac{\partial f_{\mathbf{w}}}{\partial t_1} (t,a_2,\ldots,a_k) \le 0, \forall t \in (a_1, a_2).
\]
Hence, we apply the mean value theorem to obtain
\begin{equation} \label{eq:1}
\begin{aligned}
f_{\mathbf{w}}(a_1,a_2,\ldots,a_{k-1},a_k) & \ge f_{\mathbf{w}}(a_1,a_2,\ldots,a_{k-1},a_{k-1}) \\
& = f_{w_1,w_2,\ldots,w_{k-1}+w_k}(a_1,a_2,\ldots,a_{k-1}),
\end{aligned}
\end{equation}
and
\begin{equation} \label{eq:2}
f_{\mathbf{w}}(a_1,a_2,\ldots,a_{k-1},a_k) \ge f_{w_1+w_2,\ldots,w_{k-1},w_k}(a_2,\ldots,a_{k-1},a_k).
\end{equation}
Inductively, we derive from \eqref{eq:1} and \eqref{eq:2} that
\begin{align*}
f_{\mathbf{w}}(a_1,a_2,\ldots,a_{k-1},a_k) & \ge f_{w_1+\ldots+w_l,w_{l+1}+\ldots+w_{k}}(a_l,a_{l+1})\\
& = \frac{(w_1 + w_2 + \ldots + w_l) a_l^d + (w_{l+1} + w_{l+2} + \ldots w_k) a_{l+1}^d}{((w_1 + w_2 + \ldots + w_l) a_l + (w_{l+1} + w_{l+2} + \ldots w_k) a_{l+1})^d}
\end{align*}
for $1 \leq l \leq k-1$. Define
\begin{equation*}
C_{l,\mathbf{w}}(t):=\frac{(w_1 + w_2 + \ldots + w_l) + (w_{l+1} + w_{l+2} + \ldots w_k) t^d}{((w_1 + w_2 + \ldots + w_l) + (w_{l+1} + w_{l+2} + \ldots w_k) t)^d}.
\end{equation*}
Note that $C_{l, \mathbf{w}}(t)$ is strictly increasing on $[1,\infty)$ and such that
\begin{equation*}
C_{l,\mathbf{w}}(1)=1,
\end{equation*}
and
\begin{equation*}
f_{\mathbf{w}}(a_1,a_2,\ldots,a_k) \ge C_{l,\mathbf{w}}(\frac{a_{l+1}}{a_l}), \forall 1 \le l \le k-1.
\end{equation*}
The proof is thus complete by defining the function $C(t)$ as
\begin{equation*}
C(t):=\inf_{l,\mathbf{w}} C_{l,\mathbf{w}}(t).
\end{equation*}
\end{proof}

With the introduction of Lemmas \ref{lem:universal_gap_property} and \ref{lem:holder_strict_inequality}, we are in a position of proving Theorem \ref{thm:irr_H_condition}.

\begin{proof}[Proof of Theorem \ref{thm:irr_H_condition}]
Suppose $M = m$. It follows that the spectral radius of $A$ is $\rho(A) = M$ and $x_i(n) = M^{d^1+\ldots+d^n}$ for all $i$ and $n \in \Nint$. Hence,
$$
h(\mathcal{T}_A) = \lim_{n \to \infty} \frac{\log (k M^{d + d^2 + \cdots + d^n})}{|\Delta_n|} = \log M = h(\mathsf{X}_A).
$$

Conversely, suppose $M > m$. Let $\delta \in (1,(\frac{M}{m})^{\frac{d}{d+1}})$ be fixed and let $S=S(\delta)$ be defined as in Lemma \ref{lem:universal_gap_property}. We then apply in the following a trick used in \cite[Theorem 3.3]{PS-TCS2018}. Since $A$ is irreducible, there exists a probability eigenvector $\mathbf{v}=(v_1, v_2, \ldots, v_k)$ such that $\mathbf{v}^T A = \rho(A) \mathbf{v}^T$ with $v_i > 0$ for all $1 \le i \le k$. According to Remark \ref{rmk:pigeonhole}, for every $n+1 \in S$, there exists $1 \leq l < k$ such that $\frac{(A \mathbf{x}(n-1))_{i_{l+1}}}{(A \mathbf{x}(n-1))_{i_l}} > \delta^{\frac{1}{d (k-1)}}$, and let $\mathbf{w}$ be the permutation $\{w_j\}_{j=1}^k=\{v_{i_j}\}_{j=1}^k$. It follows from Lemma \ref{lem:holder_strict_inequality} that there exists a strictly increasing function $C:[1,\infty) \to [1,\infty)$ with $C(1)=1$ satisfying
\begin{equation*}
    \frac{J_{\mathbf{w}}((A \mathbf{x}(n))_{i_1},(A \mathbf{x}(n))_{i_2},\ldots,(A \mathbf{x}(n))_{i_k})}{I_{\mathbf{w}}^d((A \mathbf{x}(n))_{i_1},(A \mathbf{x}(n))_{i_2},\ldots,(A \mathbf{x}(n))_{i_k})} \ge C\left(\frac{(A \mathbf{x}(n))_{i_{l+1}}}{(A \mathbf{x}(n))_{i_l}}\right) \ge C(\delta^{\frac{1}{d (k-1)}}).
\end{equation*}
By writing
\begin{equation*}
\gamma(n):= \begin{cases}
C(\delta^{\frac{1}{d (k-1)}}), & \text{if } n+1 \in S;\\
1 & \text{if } n+1 \notin S,
\end{cases}
\end{equation*} 
we deduce that 
\begin{equation*}
\frac{\sum_{i=1}^k v_i (A \mathbf{x}(n))_i^d}{\left(\sum_{i=1}^k v_i (A \mathbf{x}(n))_i \right)^d} = \frac{J_{\mathbf{w}}((A \mathbf{x}(n))_{i_1},(A \mathbf{x}(n))_{i_2},\ldots,(A \mathbf{x}(n))_{i_k})}{I_{\mathbf{w}}^d((A \mathbf{x}(n))_{i_1},(A \mathbf{x}(n))_{i_2},\ldots,(A \mathbf{x}(n))_{i_k})} \ge \gamma(n).
\end{equation*}
Furthermore,
\begin{align*}
& \hphantom{\ = } \sum_{i=1}^k v_i x_i(n) = \sum_{i=1}^k v_i (A \mathbf{x}(n-1))_i^d \\
& = \frac{\sum_{i=1}^k v_i (A \mathbf{x}(n-1))_i^d}{\left(\sum_{i=1}^k v_i (A \mathbf{x}(n-1))_i \right)^d} \cdot \rho(A)^d \cdot \left(\sum_{i=1}^k v_i x_i(n-1) \right)^d  \\
& \ge \gamma(n-1) \cdot\rho(A)^d \cdot \left(\sum_{i=1}^k v_i x_i(n-1) \right)^d \\
& \ge \gamma(n-1) \gamma(n-2)^{d} \cdot\rho(A)^{d+d^2} \cdot \left(\sum_{i=1}^k v_i x_i(n-2) \right)^{d^2} \\
& \ge \gamma(n-1) \gamma(n-2)^{d} \ldots \gamma(0)^{d^{n-1}}\rho(A)^{d+d^2+\cdots+d^n} \cdot \left(\sum_{i=1}^k v_i x_i(0)\right)^{d^n} \\
& = \gamma(n-1) \gamma(n-2)^{d} \ldots \gamma(0)^{d^{n-1}}\rho(A)^{d+d^2+\cdots+d^n}.
\end{align*}
Now that $|\Delta_n|=1+d+\ldots+d^n$, we take logarithm and divide both sides of the inequality above by $|\Delta_n|$ to derive
\begin{equation*}
\frac{\log \norm{B_n (\mathcal{T}_{\mathbf{A}})}}{|\Delta_n|} \ge \frac{|\Delta_n| \cdot \log\rho(A)}{|\Delta_n|} -\frac{\log\rho(A)}{|\Delta_n|} + \frac{\sum_{i \in S \cap [0,n]} d^{n-i}}{\sum_{0\le i \le n} d^i} \log C(\delta^{\frac{1}{d (k-1)}}).
\end{equation*}
By letting $n$ tend to inifinity on both sides we derive $h(\mathcal{T}_{A}) > h (X_A)$, since $S$ has positive lower Banach density.
\end{proof}

\section{Conjugacy Invariant of Topological Entropy of Tree-Shifts}

As an application of results of the previous section, we consider the topological conjugacy and topological entropy for hom tree-shifts of finite type, which are notions defined as follows. Shift spaces $X$ and $Y$ are topologically conjugate, denoted by $X \cong Y$, if there exists a one-to-one correspondence $\phi: X \to Y$ induced by a block map $\Phi: B_m(X) \to \mathcal{A}(Y)$ for some $m \in \mathbb{N}$, where $\mathcal{A}(Z)$ denotes the alphabet of a shift space $Z$. For $n \in \mathbb{N}$, the $n$th higher block shift of $X$ is
$$
X^{[n]} = \{(y_i)_{i \in \mathbb{N}} \in (\mathcal{A}(X)^{n})^{\mathbb{N}}: \text{ $\exists x \in X$ such that } y_i = x_i \cdots x_{i+n-1} \forall~i \in \mathbb{N}\}.
$$
It is easy to see that $X \cong X^{[n]}$ for all $n$ (cf.~\cite{LM-1995}). Similarly, two tree-shifts $\mathcal{T}$ and $\mathcal{S}$ are topologically conjugate if there is a one-to-one correspondence $\psi: \mathcal{T} \to \mathcal{S}$ induced by a block map $\Psi: B_m(\mathcal{T}) \to \mathcal{A}(\mathcal{S})$ for some $m \in \mathbb{N}$. Ban and Chang \cite{BC-TAMS2017} demonstrated that $\mathcal{T}^{[n]} \cong \mathcal{T}$ for each $n \in \mathbb{N}$, where
$$
\mathcal{T}^{[n]} = \{t' \in B_n(\mathcal{T})^{\Sigma^*}: \text{ there exists $t \in \mathcal{T}$ such that } t'_w = t_{w\Delta_n} \forall~w \in \Sigma^*\}
$$
is the $n$th higher block tree-shift of $\mathcal{T}$.

A subset $\mathbf{s}=\{s_i\}_{i \in  \mathbb{N}} \subset \Sigma^*$ is called a \emph{chain} if $s_1 = \epsilon$ and $s_{i+1} \in s_i \Sigma$ for all $i \geq 1$; in other words, $\mathbf{s}$ is an infinite path initiated at the root. Define $\pi_{\mathbf{s}}: \mathcal{A}^{\Sigma^*} \to \mathcal{A}^{\mathbb{N}}$ as $(\pi_{\mathbf{s}} t)_i = t_{s_i}$ for $i \in \mathbb{N}$. For each shift space $X \subseteq \mathcal{A}^{\mathbb{N}}$, the corresponding \emph{hom tree-shift} $\mathcal{T}_X$ is defined as
$$
\mathcal{T}_X = \{t \in \mathcal{A}^{\Sigma^*}: \pi_{\mathbf{s}} (t) \in X \text{ for every chain } \mathbf{s}\}.
$$
In particular, a hom tree-shift is called a \emph{hom tree-shift of finite type} (respectively, \emph{hom Markov tree-shift}) if $X$ is an SFT (respectively, nearest neighbor SFT). Although the topological entropy of tree-shifts is not conjugacy invariant in general (see \cite{Bowen-JAMS2010} for instance), Theorem \ref{thm:n_block_entropy_inv} indicates that it is invariant for hom tree higher block shifts. Notably, we call $\mathcal{T}_Y$ a \emph{hom tree higher block shift} if $Y = X^{[m]}$ is a higher block shift of some shift space $X$; a hom tree higher block shift may not be a higher block tree-shift generally.

\begin{theorem} \label{thm:n_block_entropy_inv}
Suppose $X$ is a shift space and $m \in \mathbb{N}$. Let $X^{[m]}$ be a higher block shift of $X$. Then $h(\mathcal{T}_X) = h(\mathcal{T}_{X^{[m]}})$.
\end{theorem}
\begin{proof}
Define $f: \mathcal{A}^m \to \mathcal{A}$ as $f(a_1, \ldots, a_m) = a_m$. Let $f_n: (\mathcal{A}^m)^{\Delta_n} \to \mathcal{A}^{\Delta_n}$ be defined as $(f_n(v))_w = f(v_w)$ for all $w \in \Delta_n$. We claim that $f_n(B_n(\mathcal{T}_{X^{[m]}})) \subseteq B_n(\mathcal{T}_X)$, and that $|B_n(\mathcal{T}_{X^{[m]}})| \le |\mathcal{A}|^{m-1} |B_n(\mathcal{T}_X)|$. 

Let $v \in B_n(\mathcal{T}_{X^{[m]}})$. Then, for all $s_1, \ldots, s_i \in \Sigma$, $v_\epsilon v_{s_1} v_{s_1 s_2} \ldots v_{s_1 s_2 \ldots s_i} \in B(X^{[m]})$ and $f(v_\epsilon) f(v_{s_1}) f(v_{s_1 s_2}) \ldots f(v_{s_1 s_2 \ldots s_i}) \in B(X)$. Therefore, $f_n(B_n(\mathcal{T}_{X^{[m]}})) \subseteq B_n(\mathcal{T}_X)$. Furthermore, if $v,v' \in B_n(\mathcal{T}_{X^{[m]}})$ satisfy $f_n(v)=f_n(v')$, then it is not hard to see $v=v'$ if and only if $v_{\epsilon}=(a_1 a_2 \ldots a_{m-1} a) = (b_1 b_2 \ldots b_{m-1} a)=v'_{\epsilon}$. Combining the above, we obtain $|B_n(\mathcal{T}_{X^{[m]}})| \le |\mathcal{A}|^{m-1} \cdot |B_n(\mathcal{T})|$ and 
\begin{align*}
    h(\mathcal{T}_{X^{[m]}}) & = \lim_{n \to \infty} \frac{\log |B_n(\mathcal{T}_{X^{[m]}})|}{|\Delta_n|} \\
    & \le \lim_{n \to \infty} \frac{(m-1) \log |\mathcal{A}| + \log |B_n(\mathcal{T}_X)|}{|\Delta_n|} = h(\mathcal{T}_X).
\end{align*}

On the other hand, for every $\overline{w} \in \Sigma^{m-1}$, define $g_{\overline{w},n}: B_n(\mathcal{T}_X) \to B_{n-m+1}(\mathcal{T}_{X^{[m]}})$ for $n \ge m$ as 
\[
    (g_{\overline{w},n}(u))_{w} = (u_{\overline{w}_i} \ldots u_{\overline{w}_{m-1}} u_{w_1} \ldots u_{w_i}),
\]
where $w=w_1 w_2 \ldots w_i \in \Sigma^i$ and $\overline{w}_0=\epsilon$. It is seen that 
\[
|B_n(\mathcal{T}_X)| \le {|\mathcal{A}|}^{|\Delta_{m-1}|} \sum_{\overline{w} \in \Sigma^{m-1}} |g_{\overline{w},n} (B_n(\mathcal{T}_X))| \le {|\mathcal{A}|}^{|\Delta_{m-1}|} {|B_{n-m+1}(\mathcal{T}_X)|}^{d^{m-1}}.
\]
Hence,
\begin{align*}
    h(\mathcal{T}_X) & = \lim_{n \to \infty} \frac{\log |B_n(\mathcal{T}_X)|}{|\Delta_n|} \\
    & \le \lim_{n \to \infty} \frac{|\Delta_{m-1}| \log |\mathcal{A}| + {d^{m-1}} \log |B_{n-m+1}(\mathcal{T}_X)|}{|\Delta_n|} = h(\mathcal{T}_{X^{[m]}}). 
\end{align*}
This finishes the proof.
\end{proof}

\begin{remark}\label{rmk:intro_HBS}
Petersen and Salama \cite{PS-DCDS2020} indicated that $h(\mathcal{T}_X) \geq h(X)$ for any shift space $X$. It is of our interest to study when the equality holds. Theorems \ref{thm:irr_H_condition} and \ref{thm:n_block_entropy_inv} provide a necessary and sufficient condition for the case where $X$ is an irreducible shift of finite type. More explicitly, each SFT $X$ is topologically conjugate with its higher block shift $X^{[n]}$ which is Markovian for some $n \in \mathbb{N}$ (cf.~\cite{LM-1995}). Let $A^{[n]}$ be the adjacency matrix of $X^{[n]}$. Theorem \ref{thm:irr_H_condition} indicates whether $h(\mathcal{T}_{X^{[n]}}) = h(X^{[n]})$.
\end{remark}

\begin{proposition} \label{prop:converse_conjugacy}
	Suppose $X$ and $Y$ are shift spaces. If $\mathcal{T}_{X} \cong \mathcal{T}_{Y}$, then $X \cong Y$. 
\end{proposition}
\begin{proof}
	Let $\phi: X \to Y$ and $\psi: Y \to X$ be maps induced by the block maps $\Phi:B_n(\mathcal{T}_{X}) \to \mathcal{A}(\mathcal{T}_{Y})$, $\Psi:B_m(\mathcal{T}_{Y}) \to \mathcal{A}(\mathcal{T}_{X})$ respectively, satisfying $\psi \circ \phi = Id_{\mathcal{T}_{X}}$ and $\psi \circ \phi = Id_{\mathcal{T}_{Y}}$. We first claim that if $t \in \mathcal{T}_{X}$ such that $t_w = t_s$ for all $w,s \in \Sigma^k$ and for all $k \ge 0$, then $(\phi(t))_{w} = (\phi(t))_{s}$ for all $w,s \in \Sigma^k$ and for all $k \ge 0$. Indeed, it follows from the fact $\phi$ is a sliding block code that
	\begin{equation*}
		(\phi(t))_w = (\sigma_w \phi(t))_\epsilon = (\phi(\sigma_w t))_\epsilon = (\phi(\sigma_s t))_\epsilon = (\sigma_s \phi(t))_\epsilon = (\phi(t))_s,
	\end{equation*} 
	where $\sigma_w(t)$, $w \in \Sigma^*$, denotes the shifted labeled tree defined as $(\sigma_w(t))_{s}=t_{ws}$ for every $s \in \Sigma^*$. Note that this property is also shared by $\psi$ since it is also a sliding block code.
	
	Now suppose $u=(u_0 u_1 \cdots u_n) \in B_{n+1}(X)$ (respectively, $B_{m+1}(Y)$). We denote by $\overline{u} \in B_n(\mathcal{T}_{X})$ (respectively, $B_m(X)$) the block such that $\overline{u}_w := u_{\norm{w}}$ for all $w \in \Delta_n$. We define block maps $\Phi':B_{n+1}(X) \to \mathcal{A}(Y)$, $\Psi':B_{m+1}(Y) \to \mathcal{A}(X)$ on $X$ and $Y$ respectively as
	\begin{align*}
	\Phi'(u):=\Phi(\overline{u}),
	\end{align*}
	and
	\begin{align*}
	\Psi'(u):=\Psi(\overline{u}).
	\end{align*}
	It follows from our claim that the induced sliding block code $\phi':X \to Y$ and $\psi':Y \to X$ of $\Phi'$ and $\Psi'$ are well-defined, both injective, $\psi \circ \phi = Id_X$ and $\phi \circ \psi = Id_Y$. This finishes the proof.
\end{proof}

Proposition \ref{prop:converse_conjugacy} demonstrates that $\mathcal{T}_X \cong \mathcal{T}_Y$ is followed by $X \cong Y$. The next proposition indicates further that, not only the inverse is false in general, but there is a Markov shift such that its associated hom tree-shift is not topologically conjugate to its hom tree higher block shift. To prove the proposition, we exploit the irreducibity of a tree-shift. More specifically, a tree-shift $\mathcal{T}$ is said to be \emph{irreducible} if for all $n$-blocks $u$ and $v$ of $\mathcal{T}$, there exists $t \in \mathcal{T}$ and $w \in \Sigma^\ast \setminus \Delta_n$ such that $t|_{\Delta_n}=u$ and that $(\sigma_w t)|_{\Delta_n}=v$. It is not hard to see that irreducibility is a conjugacy-invariant property for tree-shifts.

\begin{proposition} \label{prop:homTSFT-higher-block-not-conjugate}
There exist an SFT $X$ and $m \in \mathbb{N}$ such that $ \mathcal{T}_X \ncong  \mathcal{T}_{X^{[m]}}$.
\end{proposition}
\begin{proof}
Let $X=\mathsf{X}_{\mathcal{F}} \subset \{0,1\}^{\mathbb{N}}$ with $\mathcal{F}=\{000,011,110,101,111\}$ and $Y=\mathsf{X}_{\mathcal{F}'}\subset \{a,b,c\}^{\mathbb{N}}$ with $\mathcal{F}'=\{aa,ac,ba,bb,cb,cc\}$, i.e., $X=\{(100)^{\infty},(001)^{\infty},(010)^{\infty}\}$ and $Y=\{(abc)^{\infty},(bca)^{\infty},(cab)^{\infty}\}$. Note that $Y = X^{[2]}$ by a 2-block map $\Phi:B_2 (X) \to \{a,b,c\}$ defined as $\Phi(00)=a$, $\Phi(01)=b$, $\Phi(10)=c$. Since $X$ is irreducible, $Y$ is also irreducible, which is equivalent to irreducibility of $\mathcal{T}_Y$ by \cite[Theorem 3.3]{Ban}. However, for the admissible blocks
\begin{equation*}
	u:=\vcenter{\hbox{\begin{tikzpicture}[
	scale = 0.6, transform shape, thick,
	every node/.style = {minimum size = 5mm},
	grow = down,  
	level 1/.style = {sibling distance=2cm},
	level 2/.style = {sibling distance=1cm}, 
	level distance = 1.5cm]
	\node[]  (a) {1} [grow'=down]
	child {node[] {0}
		child{node[] {0}}
		child{node[] {0}}
	}
	child {node[] {0}
	child{node[] {0}}
		child{node[] {0}}
	};
	\end{tikzpicture}}}, v:=\vcenter{\hbox{\begin{tikzpicture}[
	scale = 0.6, transform shape, thick,
	every node/.style = {minimum size = 5mm},
	grow = down,  
	level 1/.style = {sibling distance=2cm},
	level 2/.style = {sibling distance=1cm}, 
	level distance = 1.5cm]
	\node[]  (a) {0} [grow'=down]
	child {node[] {0}
		child{node[] {1}}
		child{node[] {1}}
	}
	child {node[] {1}
	child{node[] {0}}
		child{node[] {0}}
	};
	\end{tikzpicture}}},
\end{equation*}
it is not hard to see there exists no $t \in \mathcal{T}_Y$ and $w \in \Sigma^\ast \setminus \Delta_n$ such that $t|_{\Delta_n}=u$ and that $(\sigma_w t)|_{\Delta_n}=v$. Since irreducibility is a conjugacy-invariant property, the lack of such labeled tree $t$ illustrates that $\mathcal{T}_X$ is not irreducible. This leads to a contradiction and hence $\mathcal{T}_{Y} \ncong \mathcal{T}_{X}$.
\end{proof}

Proposition \ref{prop:counterexamples} shows that, even for hom tree-shifts, topological entropy is not conjugacy invariant in general.

\begin{proposition}\label{prop:counterexamples}
There exist SFTs $X$ and $Y$ such that $X \cong Y$ but $h(\mathcal{T}_X) \neq h(\mathcal{T}_Y)$.
\end{proposition}
\begin{proof}
Let $A, B$ be given as
$$
A = \begin{pmatrix}
1 &1 &0 \\
0 &0 &1 \\
1 &1 &1
\end{pmatrix}
\quad \text{and} \quad
B = \begin{pmatrix}
1 &1 \\
1 &1
\end{pmatrix}.
$$
Observe that $A$ is strongly shift equivalent to $B$ since $A = RS$ and $B = SR$, where
$$
R = \begin{pmatrix}
1 &0 \\
0 &1 \\
1 &1
\end{pmatrix}
\quad \text{and} \quad
S = \begin{pmatrix}
1 &1 &0 \\
0 &0 &1
\end{pmatrix}.
$$
Hence, $\mathsf{X}_A \cong \mathsf{X}_B$. However, Theorem \ref{thm:irr_H_condition} demonstrates that
$$
h(\mathcal{T}_A) > h(\mathsf{X}_A) = h(\mathsf{X}_B) = h(\mathcal{T}_B).
$$
\end{proof}

\section{Further Discussion and Conclusion}

Aside from the discussion of conjugacy invariant of topological entropy for hom tree-SFTs, there are fruitful novel phenomena that have not been seen in $\mathbb{Z}^d$ shift spaces. This section illustrates, last but not least, two of our observations. Recall that each SFT is conjugate with a Markov shift that is its higher block shift (Remark \ref{rmk:intro_HBS}), this section considers only Markov shifts and hom Markov tree-shifts without loss of generality.

\subsection{Reducible Tree-Shifts of Finite Type}

Suppose $A$ is reducible with irreducible components $A_1, \ldots, A_r$. It is well-known that $h(\mathsf{X}_A) = \sup h(\mathsf{X}_{A_i})$ (see \cite{LM-1995} for more details). Besides, the collection of topological entropy of SFTs over $\mathcal{A}$ is dense in $[0, \log |\mathcal{A}|]$, we refer to \cite{Des-IM2006} for more general results. This section reveals that both the properties fail for trees due to the topological structure of trees.

A binary matrix $A$ is called \emph{essential} if $A$ has no zero rows and zero columns. Every nonzero matrix $A \in \{0, 1\}^{k \times k}$ can be reduced to its associated essential matrix $A^e \in \{0, 1\}^{l \times l}$, $1 \le l \le k$, by repeating the following process whenever it is needed: if the $i$th row or column of $A$ is zero, then delete the $i$th row and $i$th column of $A$. It is easily seen that $h(\mathcal{T}_A)=h(\mathcal{T}_{A^e})$. Hence, in studying $h(\mathcal{T}_A)$, we always assume that $A$ is essential.

The following proposition shows that for any $A \in \{0, 1\}^{k \times k}$, $h(\mathcal{T}_A)$ is either $0$ or not less than $\frac{\log 2}{2}$. It follows immediately that the set of topological entropy of hom tree-SFTs is not dense in $[0, \log |\mathcal{A}|]$.

\begin{proposition} \label{prop:entropy_gap}
	Suppose $A$ is an essential binary matrix and $A \neq 0$. Let $M$ be the maximal row sum of $A$.
\item[(i)] If $M = 1$, then $h(\mathcal{T}_A) = 0$.
\item[(ii)] If $M \ge 2$, then $h(\mathcal{T}_A) \ge \frac{d-1}{d} \log M \ge \frac{1}{2} \log 2$.
\end{proposition}
\begin{proof}
(i) If $M = 1$, it is cleat that $|B_n(\mathcal{T}_A)| \le k$. Then, $h(\mathcal{T}_A) = 0$ follows immediately.

(ii) If $M \ge 2$, there exist $1 \le i' \le k$ and $1 \le j_1 < j_2 < \ldots < j_M \le k$  such that $\sum_{j=1}^k a_{i',j} = \sum_{l=1}^M a_{i',j_l} = M$. Then, for any pattern $u:\Sigma^n \to \{j_1, j_2, \ldots, j_M\}$, there exists an $n$-block $u': \Delta_n \to \mathcal{A}$ such that $u'_w=u_w$ for all $w \in \Sigma^n$. More specifically, since $A$ is essential, there exists an $(n-1)$-block $u'' \in B_{n-1}(\mathcal{T}_A)$ such that $u''_w = i'$ for every $w \in \Sigma^{n-1}$, and the desired $u'$ can be obtained through a proper extension of $u''$. Therefore, $|B_n(\mathcal{T}_A)| \ge M^{d^n}$, which implies $h(\mathcal{T}_A) \ge \frac{d-1}{d} \log M$. The proof is complete. 
\end{proof}

\begin{problem}
Is the set of topological entropy of hom tree-SFTs over $\mathcal{A}$ dense in $[\frac{d-1}{d} \log 2, \log |\mathcal{A}|]$?
\end{problem}

The following proposition, as an application of Proposition \ref{prop:entropy_gap}, shows that the topological entropy of a reducible hom Markov tree-shift could be larger than that of any of its irreducible components.

\begin{proposition}\label{prop:entropy_large_irred_component}
Suppose $A$ is reducible with irreducible components $A_1, \ldots, A_r$. Then $h(\mathcal{T}_A) \geq \sup \{h(\mathcal{T}_{A_i})\}$, and the equality does not hold in general.
\end{proposition}
\begin{proof}
Obviously, $h(\mathcal{T}_A) \geq \sup \{h(\mathcal{T}_{A_i})\}$ since $|B_n(\mathcal{T}_A)| \geq |B_n(\mathcal{T}_{A_i})|$ for all $n \in \mathbb{N}$ and $1 \leq i \leq r$.

Consider $A = \begin{bmatrix} 1 & 1 & 1 & 1\\ 1 & 0 & 1 & 1\\0 & 0 & 1 & 1\\0 & 0 & 1 & 0  \end{bmatrix}$ with the irreducible components $A' = A_1 = A_2 = \begin{bmatrix} 1 & 1\\1 & 0  \end{bmatrix}$ on $2$-tree. Proposition \ref{prop:entropy_gap} (ii) indicates that $h(\mathcal{T}_A) \ge \log 2$. Petersen and Salama \cite{PS-TCS2018} proved that
\[h(\mathcal{T}_{A'})  = \inf_n \frac{\log |B_n(\mathcal{T}_{A'})|}{|\Delta_n|} \le \frac{\log |B_1(\mathcal{T}_{A'})|}{3} = \frac{\log 5}{3} < \log 2 \le h(\mathcal{T}_A).
\]
So in general, for reducible matrix $A$, $h(\mathcal{T}_A) \ge \sup h(\mathcal{T}_{A_i})$, where $A_i$ denotes the $i$-th irreducible component of the system defined by $A$.
\end{proof}

On the other hand, there is a nontrivial example such that $h(\mathcal{T}_A) = \sup \{h(\mathcal{T}_{A_i})\}$, where $A_1. A_2, \ldots, A_r$ are the irreducible components of $A$.

\begin{example}
Let $A = \begin{bmatrix} 1 & 1 & 0\\ 0 & 1 & 1\\0 & 1 & 1  \end{bmatrix}$. It can be calculated in the same manner as in the proof of Theorem \ref{thm:irr_H_condition} that $h(\mathcal{T}_{A})=\log 2$ while $A$ is an reducible matrix.
\end{example}

\begin{problem}
Suppose $A$ is reducible with irreducible components $A_1, \ldots, A_r$. Under what condition the equality $h(\mathcal{T}_{A}) = \sup \{h(\mathcal{T}_{A_i})\}$ holds?
\end{problem}

\subsection{Conclusion}

This paper focuses on the topological entropy of hom tree-SFTs. In the following list of main results of this elucidation, many of them, rather than consequences, are just the beginning of further investigation.
\begin{enumerate}
    \item Suppose $\mathcal{T}_X$ is the hom tree-SFT induced by the SFT $X$. Then $h(\mathcal{T}_X) = h(X)$ if and only if $\sum_j A_{i, j} = \sum_j A_{i', j}$ for all $i, i'$, where $A$ is the adjacency matrix of the graph representation of $X$ (Theorems \ref{thm:irr_H_condition} and \ref{thm:n_block_entropy_inv}.)
    \item Suppose $X$ and $Y$ are shift spaces. Then $\mathcal{T}_X \cong \mathcal{T}_Y$ is sufficient but not necessary for $X \cong Y$ (Propositions \ref{prop:converse_conjugacy} and \ref{prop:homTSFT-higher-block-not-conjugate}).
    \item The set $\{h(\mathcal{T}_X)\}$ of topological entropy of hom tree-SFTs is not dense in $[0, \log |\mathcal{A}|]$ (Proposition \ref{prop:entropy_gap}).
    \item There is a hom Markov tree-shift that has strictly larger topological entropy than shifts formed by any of its irreducible components (Proposition \ref{prop:entropy_large_irred_component}).
\end{enumerate}

\subsection*{Acknowledgements}
We thank Professors Karl Petersen and Ibrahim Salama for helpful comments on these topics, and we appreciate the comments from the reviewers for they greatly improve the readability of this article.

\bibliographystyle{amsplain}
\bibliography{grece}

\end{document}